\documentclass[]{amsart}

\usepackage[english]{babel}
\usepackage[utf8]{inputenc}
\usepackage[T1]{fontenc}
\usepackage{lmodern}
\usepackage{microtype}
\usepackage{amsmath, amsfonts, amssymb, amsthm}
\usepackage{mathtools}
\usepackage{color}
\usepackage[normalem]{ulem}
\usepackage{tikz}
\usepackage{float}
\usepackage[labelfont=sc]{caption}

\newcommand{\CC}{\mathbb{C}}

\newcommand{\GG}{\mathbb{G}}
\newcommand{\NN}{\mathbb{N}}

\newcommand{\QQ}{\mathbb{Q}}
\newcommand{\RR}{\mathbb{R}}
\newcommand{\ZZ}{\mathbb{Z}}

\newcommand{\D}{\mathcal{D}}

\newcommand{\F}{\mathcal{F}}
\newcommand{\Hc}{\mathcal{H}}
\newcommand{\I}{\mathcal{I}}
\newcommand{\J}{\mathcal{J}}

\newcommand{\Oc}{\mathcal{O}}
\newcommand{\Pc}{\mathcal{P}}

\newcommand{\Sc}{\mathcal{S}}
\newcommand{\T}{\mathcal{T}}
\newcommand{\X}{\mathcal{X}}

\newcommand{\set}[1]{\left\{ #1 \right\}}
\newcommand{\setb}[1]{\left( #1 \right)}
\newcommand{\abs}[1]{\left| #1 \right|}

\newtheorem{mymasterthm}{notForUse}
\theoremstyle{definition}

\newtheorem{remark}[mymasterthm]{Remark}

\theoremstyle{plain}
\newtheorem{lemma}[mymasterthm]{Lemma}
\newtheorem{theorem}[mymasterthm]{Theorem}
\newtheorem{corollary}[mymasterthm]{Corollary}
\newtheorem{proposition}[mymasterthm]{Proposition}

\allowdisplaybreaks

\title[Decidability of matrix equations and related Diophantine problems]{Decidability of multiplicative matrix equations\\and related Diophantine problems}
\makeatletter
\@namedef{subjclassname@2020}{\textup{2020} Mathematics Subject Classification}
\makeatother

\author[S. Heintze]{Sebastian Heintze}
\address{Sebastian Heintze,
	Graz University of Technology,
	Institute of Analysis and Number Theory,
	Steyrergasse 30/II,
	A-8010 Graz, Austria}
\email{heintze@math.tugraz.at}

\author[A. Noubissie]{Armand Noubissie}
\address{Armand Noubissie,
	Graz University of Technology,
	Institute of Analysis and Number Theory,
	Kopernikusgasse 24/II,
	A-8010 Graz, Austria}
\email{armand.noubissie@tugraz.at}

\author[R.F. Tichy]{Robert F. Tichy}
\address{Robert F. Tichy,
	Graz University of Technology,
	Institute of Analysis and Number Theory,
	Kopernikusgasse 24/II,
	A-8010 Graz, Austria}
\email{tichy@tugraz.at}

\begin{document}
	
	\begin{abstract}
		Some new decidability results for multiplicative matrix equations over algebraic number fields are established.
		In particular, special instances of the so-called knapsack problem are considered.
		The proofs are based on effective methods for Diophantine problems in finitely generated domains as presented in the recent book of Evertse and Gy\H{o}ry.
		The focus lies on explicit bounds for the size of the solutions in terms of heights as well as on bounds for the number of solutions.
		This approach also works for systems of symmetric matrices which do not form a semigroup.
		In the final section some related counting problems are investigated.
	\end{abstract}
	
	\maketitle
	
	
	\section{Introduction}
	
	By a classical result of Markov \cite{markov-1951} it cannot be decided by a Turing machine if a matrix $M$ belongs to a given semigroup of $(6 \times 6)$-integer-valued matrices (membership problem).
	Mihailova \cite{mihailova-1958} extended this undecidability result to subgroups of $\mathrm{SL}_4(\ZZ)$, and later Paterson \cite{paterson-1970} proved that the membership problem is undecidable for semigroups of integer-valued $(3 \times 3)$-matrices.
	It is still an open question whether any membership problem is decidable in $\mathrm{SL}_3(\ZZ)$.
	A case of particular interest is concerned with symmetric matrices since they do not form a semigroup.
	For $(2 \times 2)$-matrices often such membership problems are decidable.
	For instance, Potapov and Semukhin \cite{potapov-semukhin-2017} proved that the semigroup membership problem is decidable in $\mathrm{GL}_2(\ZZ)$.
	Further results can be found in \cite{bell-potapov-2008,bell-potapov-2012}.
	However, these problems are still open for semigroups of arbitrary integer matrices or for subgroups of $\mathrm{GL}_2(\QQ)$.
	
	A related membership problem is the so-called knapsack problem.
	This problem is concerned with the question whether for given $(n \times n)$-square matrices $A_1,\ldots,A_s$ and $A$ there exist positive integers $k_1,\ldots,k_s$ such that
	\begin{equation}
		\label{eq:intromember}
		A_1^{k_1} \cdot \ldots \cdot A_s^{k_s} = A.
	\end{equation}
	The special case when $A=0$ (zero matrix) is called the mortality problem, and the case $A=I$ the identity problem.
	In \cite{bell-halava-harju-karhumaeki-potapov-2008} it was shown that the mortality problem is undecidable for integer matrices of large $n$ and $s$.
	The proof makes use of a reduction to Hilbert's tenth problem which was shown to be undecidable by Matijasevic in 1970 based on work of Julia Robinson and others in the 1960's.
	In \cite{cassaigne-halava-harju-nicolas-} some undecidability bounds in the case of general mortality problems were obtained by more combinatorial methods.
	
	However, in special cases the knapsack problem is decidable, for instance for the Heisenberg group $\mathfrak{H}(3,\ZZ)$ and generalizations (see the useful survey \cite{lohrey-2015}).
	The Heisenberg group is important in harmonic analysis and mathematical physics and consists (in the integer-valued case) of matrices of the form
	\begin{equation*}
		\begin{pmatrix}
			1 & a & c \\
			0 & 1 & b \\
			0 & 0 & 1
		\end{pmatrix}
	\end{equation*}
	with $a,b,c \in \ZZ$.
	In \cite{colcombet-ouaknine-semukhin-worrell-2019} it is shown that the semigroup membership problem in $\mathfrak{H}(3,\ZZ)$ is decidable.
	
	The special case of the mortality problem with $s=3$ is called the $ABC$-problem since it is concerned with equations
	\begin{equation*}
		A^{k_1} B^{k_2} C^{k_3} = 0
	\end{equation*}
	for given $(n \times n)$-matrices $A,B,C$.
	For matrices with entries from $\QQ$ or from an algebraic number field Bell, Potapov and Semukhin \cite{bell-potapov-semukhin-2021} showed that the $ABC$-problem is equivalent to the Skolem problem for linear recurrence sequences of order $n$.
	The Skolem problem is a well-known topic in Diophantine number theory and it is devoted to the question to decide for a given linear recurrence sequence $(u_j)_{j=0}^{\infty}$ whether there is $j$ such that $u_j=0$.
	Mignotte, Shorey and Tijdeman \cite{mignotte-shorey-tijdeman-1984} and Vereshchagin \cite{vereshchagin-1985} proved that the Skolem problem is decidable for linear recurrences up to order $3$ over algebraic numbers and up to order $4$ for linear recurrences over real algebraic numbers.
	The proofs rely on Baker's method obtaining effective lower bounds for linear forms in logarithms of algebraic numbers.
	More recently, Bell, Potapov and Semukhin \cite{bell-potapov-semukhin-2021} investigated the mortality problem in the case $s=4$ and proved decidability for $(2 \times 2)$-rational upper triangle matrices (again applying Baker's effective method).
	For more information concerning the Skolem problem and decidability we refer to \cite{luca-maynard-noubissie-ouaknine-worrell-,luca-ouaknine-worrell-2021,luca-ouaknine-worrell-2022-1,luca-ouaknine-worrell-2022-2,ouaknine-worrell-2015,sha-2019}.
	
	Another important special case of the knapsack problem deals with commuting matrices.
	In \cite{babai-beals-cai-ivanyos-luks-1996} it is proven that for commuting matrices the knapsack problem is decidable and the solutions $k_1,\ldots,k_s$ can be computed in polynomial time.
	However, no explicit bounds for the solutions are given.
	It is the aim of this paper to establish explicit bounds for the solutions of matrix equations of type \eqref{eq:intromember}.
	Our approach is mainly based on effective methods for Diophantine problems, as developed in the recent monograph by Evertse and Gy\H{o}ry \cite{evertse-gyory-2022}.
	The method mostly works for matrix equations over finitely generated domains, thus in paricular for algebraic number fields.
	Section~\ref{sec:2specialcases} is devoted to (strong) bounds in terms of heights for the solutions, provided that such solutions exist.
	This yields logarithmic complexity bounds for the identity problem in the case of commuting matrices.
	In this section we also obtain polynomial bounds for Heisenberg matrices.
	Since symmetric matrices do not form a semigroup we point our focus on this important class of matrices.
	Section~\ref{sec:3abc} is devoted to the $ABC$-identity-problem
	\begin{equation*}
		A^{k_1} B^{k_2} C^{k_3} = I_2
	\end{equation*}
	where $ A,B,C $ are symmetric $(2 \times 2)$-matrices.
	In Section~\ref{sec:4decida} a quite general decidability problem is investigated establishing explicit complexity bounds.
	In the final Section~\ref{sec:5counting} we prove some special counting results for multiplicatively dependent tuples of matrices, which is in the spirit of Igor Shparlinski and followers.
	In the case of integer-valued matrices we refer to the recent paper \cite{habegger-ostafe-shparlinski-} where the authors study integer matrices with given characteristic polynomial.
	Our results are weaker but the approach works in more general cases, for instance for symmetric matrices over algebraic numbers.

	\section{Commuting matrices and other special cases}
	\label{sec:2specialcases}
	
	Our first result gives a sharp logarithmic bound on the exponents appearing in a perfect power equality of two matrices.
	For this purpose we need the concept of heights.
	Given a matrix $ A $ with algebraic number entries, we denote by $ \Hc(A) $ the multiplicative height of the matrix $ A $ as a point in the projective space.
	Furthermore, $ \ll $ denotes in the following the usual Vinogradov symbol which may depend on the involved number field $ K $.
	
	\begin{theorem}
		\label{thm:case2}
		Let $ K $ be a number field. Let $ A_1 $ and $ A_2 $ be two diagonalizable matrices in $ \mathrm{GL}_n(K) $ of bounded height $ \Hc(A_i) \ll H $ for $ i \in \set{1,2} $ which satisfy a multiplicative relation
		\begin{equation}
			\label{eq:multrel}
			A_1^{k_1} = A_2^{k_2}
		\end{equation}
		with $ k_1 k_2 \neq 0 $. Then there are exponents with $ \abs{k_1},\abs{k_2} \ll \log H $ such that \eqref{eq:multrel} holds.
	\end{theorem}
	
	We will use in the proof the following lemma which, simply speaking, is a version for diagonal matrices:
	
	\begin{lemma}
		\label{lem:LvdP-general}
		Let $ \lambda_{1,1},\ldots,\lambda_{n,1},\ldots,\lambda_{1,s},\ldots,\lambda_{n,s} $ be given algebraic numbers and $ H $ an upper bound for their height. Assume that there exist non-zero integers $ k_1,\ldots,k_s $ such that
		\begin{equation*}
			\begin{pmatrix}
				\lambda_{1,1}^{k_1} \cdot \ldots \cdot \lambda_{1,s}^{k_s} &  & 0 \\
				& \ddots & \\
				0 &  & \lambda_{n,1}^{k_1} \cdot \ldots \cdot \lambda_{n,s}^{k_s}
			\end{pmatrix}
			= I_n,
		\end{equation*}
		where $ I_n $ is the identity matrix of dimension $ n $.
		Then there is such a relation with the property $ \abs{k_1},\ldots,\abs{k_s} \ll (\log H)^R $ for a positive integer $ R < s $.
	\end{lemma}
	
	\begin{proof}
		There is a finite set $ S $ of places such that $ \lambda_{1,1},\ldots,\lambda_{n,1},\ldots,\lambda_{1,s},\ldots,\lambda_{n,s} $ are all $ S $-units. By the well-known generalization of Dirichlet's unit theorem to $ S $-units (see e.g.\ Theorem~1.5.13 in \cite{bombieri-gubler-2006}), there is a finite number of fundamental $ S $-units together with a root of unity such that any $ S $-unit has a unique representation as a product of powers of these fundamental elements.
		By writing each $ \lambda_{i,j} $ in its representation, we get from the exponents a homogenous system of linear equations in $ k_1,\ldots,k_s $ with integer coefficients.
		Note that the height of these integer coefficients, coming from the exponents, is $ \ll \log H $ (confer the results in \cite{loxton-vanderpoorten-1977}).
		Since, by assumption, we have a non-zero solution to this system, we can apply Corollary~2.9.9 in \cite{bombieri-gubler-2006} and are done.
	\end{proof}
	
	\begin{proof}[Proof of Theorem \ref{thm:case2}]
		First, we get invertible matrices $ T_1 $ and $ T_2 $ as well as eigenvalues $ \lambda_{1,1},\ldots,\lambda_{n,1} $, $ \lambda_{1,2},\ldots,\lambda_{n,2} $ such that
		\begin{equation*}
			A_1 = T_1 D_1 T_1^{-1}
			\qquad \text{and} \qquad
			A_2 = T_2 D_2 T_2^{-1}
		\end{equation*}
		with
		\begin{equation*}
			D_j =
			\begin{pmatrix}
				\lambda_{1,j} &  & 0 \\
				& \ddots & \\
				0 &  & \lambda_{n,j}
			\end{pmatrix}
		\end{equation*}
		for $ j \in \set{1,2} $.
		Note that we also have $ \lambda_{i,j} \neq 0 $ and $ \Hc(\lambda_{i,j}) \ll H^{\gamma} $ for a constant $ \gamma $, depending only on the dimension of the matrices, and $ i \in \set{1,\ldots,n} $, $ j \in \set{1,2} $.
		
		Now $ A_1^{k_1} = A_2^{k_2} $ implies
		\begin{equation*}
			D_1^{k_1} = T_1^{-1} T_2 D_2^{k_2} \left( T_1^{-1} T_2 \right)^{-1}.
		\end{equation*}
		Thus $ D_1^{k_1} $ and $ D_2^{k_2} $ are similar matrices and the transformation $ T_1^{-1} T_2 $ can only permute the diagonal entries. Hence there is a permutation $ \sigma $ of $ \set{1,\ldots,n} $ such that
		\begin{equation*}
			\lambda_{i,1}^{k_1} = \lambda_{\sigma(i),2}^{k_2}
		\end{equation*}
		for $ i \in \set{1,\ldots,n} $.
		Written in matrix form, this is
		\begin{equation*}
			\begin{pmatrix}
				\lambda_{1,1}^{k_1} \lambda_{\sigma(1),2}^{-k_2} &  & 0 \\
				& \ddots & \\
				0 &  & \lambda_{n,1}^{k_1} \lambda_{\sigma(n),2}^{-k_2}
			\end{pmatrix}
			= I_n.
		\end{equation*}
		Therefore we can apply Lemma~\ref{lem:LvdP-general} for $ s=2 $, hence $ R=1 $, and are done.
	\end{proof}
	
	\begin{remark}
		Theorem~\ref{thm:case2} is true in a much stronger form if one exponent is zero. If $ A $ is an arbitrary matrix in $ \CC^{n \times n} $ satisfying $ A^k = I_n $ for an integer $ k > 0 $, then we obtain that $ k $ is bounded by a constant only depending on $ n $. This follows immediately from the following lemma.
	\end{remark}
	
	\begin{lemma}
		\label{lem:aki}
		Assume that $ A \in \CC^{n \times n} $ satisfies $ A^k = I_n $ for an integer $ k > 0 $. Then $ A $ is diagonalizable and all eigenvalues are roots of unity of degree at most $ n $.
	\end{lemma}
	
	\begin{proof}
		There exists a matrix $ T \in \mathrm{GL}_n(\CC) $ as well as a matrix $ J $ in Jordan normal form such that $ A = TJT^{-1} $. Note that $ A $ and $ J $ have the same eigenvalues. From $ A^k = I_n $ we get $ J^k = I_n $.
		Since the $ k $-th power of a Jordan block
		\begin{equation*}
			\begin{pmatrix}
				z & 1 & 0 & \cdots & 0 \\
				0 & z & 1 &  & \vdots \\
				0 & 0 & z & \ddots & 0 \\
				\vdots &  & \ddots & \ddots & 1 \\
				0 & 0 & \cdots & 0 & z
			\end{pmatrix}
		\end{equation*}
		is given by
		\begin{equation*}
			\begin{pmatrix}
				z^k & kz^{k-1} &  &  & * \\
				0 & z^k & kz^{k-1} &  & \\
				0 & 0 & z^k & \ddots & \\
				\vdots &  & \ddots & \ddots & kz^{k-1} \\
				0 & 0 & \cdots & 0 & z^k
			\end{pmatrix}
		\end{equation*}
		the matrix $ J^k $ can only be the identity matrix if $ J $ is a diagonal matrix.
		Thus $ A $ is diagonalizable.
		Furthermore, each eigenvalue $ \lambda $ of $ A $ satisfies $ \lambda^k = 1 $ and thus is a root of unity. As the eigenvalues are the roots of the characteristic polynomial which has degree $ n $, the statement follows.
	\end{proof}
	
	For products of more than two factors we need a further condition to prove a logarithmic height bound:
	
	\begin{theorem}
		\label{thm:cases}
		Let $ K $ be a number field. Let $ A_1,\ldots,A_s $ be commuting, diagonalizable matrices in $ \mathrm{GL}_n(K) $ of bounded height $ \Hc(A_i) \ll H $ for $ i \in \set{1,\ldots,s} $ which satisfy a multiplicative relation
		\begin{equation}
			\label{eq:multdep}
			A_1^{k_1} \cdot \ldots \cdot A_s^{k_s} = I_n
		\end{equation}
		with $ k_1 \cdots k_s \neq 0 $. Then there are exponents with $ \abs{k_1},\ldots,\abs{k_s} \ll (\log H)^R $ for a positive integer $ R < s $ such that \eqref{eq:multdep} holds.
	\end{theorem}
	
	We use the following well-known fact from linear algebra which can be found e.g.\ as Theorem~1.3.21 in \cite{horn-johnson-2013}:
	
	\begin{lemma}
		\label{lem:simul-diag}
		Let $ (A_f)_{f \in \F} $ be a family of commuting, diagonalizable quadratic matrices. Then they can be diagonalized with the same invertible transformation matrix $ T $, i.e.
		\begin{equation*}
			T A_f T^{-1} = D_f
		\end{equation*}
		for all $ f \in \F $.
	\end{lemma}
	
	\begin{proof}[Proof of Theorem \ref{thm:cases}]
		Since the matrices $ A_i $ commute with each other, we can bring the factors with negative exponents to the other side and assume without loss of generality that we have
		\begin{equation*}
			A_1^{k_1} \cdot \ldots \cdot A_q^{k_q} = A_{q+1}^{k_{q+1}} \cdot \ldots \cdot A_s^{k_s},
		\end{equation*}
		where all exponents are positive.
		Denote by $ \lambda_{1,i},\ldots,\lambda_{n,i} $ the eigenvalues of $ A_i $ for all $ i \in \set{1,\ldots,s} $.
		By assumption, the $ A_i $ commute with each other and thus the same holds for powers of them.
		Now we iteratively apply Lemma~\ref{lem:simul-diag} to the commuting matrices being next to each other and end up with
		\begin{equation*}
			\small
			\begin{pmatrix}
				\lambda_{1,1}^{k_1} \cdot \ldots \cdot \lambda_{1,q}^{k_q} &  & 0 \\
				& \ddots & \\
				0 &  & \lambda_{n,1}^{k_1} \cdot \ldots \cdot \lambda_{n,q}^{k_q}
			\end{pmatrix}
			=
			\begin{pmatrix}
				\lambda_{1,q+1}^{k_{q+1}} \cdot \ldots \cdot \lambda_{1,s}^{k_s} &  & 0 \\
				& \ddots & \\
				0 &  & \lambda_{n,q+1}^{k_{q+1}} \cdot \ldots \cdot \lambda_{n,s}^{k_s}
			\end{pmatrix}.
		\end{equation*}
		Then the result follows from Lemma~\ref{lem:LvdP-general}.
	\end{proof}
	
	Next, we are interested in giving bounds in the case that the condition of being diagonalizable is replaced by being a Heisenberg matrix, a classical example of non-diagonalizable matrices.
	
	\begin{theorem}
		\label{thm:heisenberg3}
		Let $ A_1 $, $ A_2 $ and $ A_3 $ be Heisenberg matrices of bounded height, i.e.\ for $ i \in \set{1,2,3} $ we have algebraic numbers $ a_i,b_i,c_i $ such that
		\begin{equation*}
			A_i =
			\begin{pmatrix}
				1 & a_i & c_i \\
				0 & 1 & b_i \\
				0 & 0 & 1
			\end{pmatrix}
		\end{equation*}
		and $ \Hc(A_i) \ll H $.
		If there are non-zero integers $ k_1 $, $ k_2 $ and $ k_3 $ such that
		\begin{equation}
			\label{eq:hbmat3}
			A_1^{k_1} A_2^{k_2} A_3^{k_3} = I_3,
		\end{equation}
		then there are exponents with $ \abs{k_1},\abs{k_2},\abs{k_3} \ll H^{58} $ such that \eqref{eq:hbmat3} holds.
	\end{theorem}
	
	\begin{proof}
		Calculating the product on the left hand side of \eqref{eq:hbmat3} and simplifying the arising equations yields the system
		\begin{align*}
			a_1 k_1 + a_2 k_2 + a_3 k_3 &= 0 \\
			b_1 k_1 + b_2 k_2 + b_3 k_3 &= 0 \\
			c_1 k_1 - \frac{1}{2} a_1 b_1 k_1 + c_2 k_2 - \frac{1}{2} a_2 b_2 k_2 + c_3 k_3 - \frac{1}{2} a_3 b_3 k_3 + \frac{1}{2} (a_2 b_3 - a_3 b_2) k_2 k_3 &= 0.
		\end{align*}
		If the vectors $ (a_1,a_2,a_3) $ and $ (b_1,b_2,b_3) $ are linearly dependent, then the summand $ \frac{1}{2} (a_2 b_3 - a_3 b_2) k_2 k_3 $ vanishes in the above system, which in this case becomes a linear system. To this we can apply Corollary~2.9.9 in \cite{bombieri-gubler-2006} and are done.
		Hence we may now assume that the vectors $ (a_1,a_2,a_3) $ and $ (b_1,b_2,b_3) $ are linearly independent.
		Note that the first two equations from the system above imply
		\begin{equation*}
			(a_2 b_3 - a_3 b_2) k_2 k_3 = (a_3 b_1 - a_1 b_3) k_1 k_3 = (a_1 b_2 - a_2 b_1) k_1 k_2.
		\end{equation*}
		Thus we can assume that $ (a_2 b_3 - a_3 b_2) (a_3 b_1 - a_1 b_3) (a_1 b_2 - a_2 b_1) \neq 0 $ since otherwise the system of equations becomes linear, which was already handled before.
		Putting $ k_1 = t $, the subsystem
		\begin{align*}
			a_2 k_2 + a_3 k_3 &= - a_1 t \\
			b_2 k_2 + b_3 k_3 &= - b_1 t
		\end{align*}
		has a unique solution $ (k_2,k_3) $, depending on $ t $, due to the linear independence of $ (a_1,a_2,a_3) $ and $ (b_1,b_2,b_3) $ as well as the assumption that for at least one non-zero value of $ t $ there is a solution at all. This solution is given by
		\begin{equation*}
			(k_1,k_2,k_3) = \left( t, \frac{a_3 b_1 - a_1 b_3}{a_2 b_3 - a_3 b_2} t, \frac{a_1 b_2 - a_2 b_1}{a_2 b_3 - a_3 b_2} t \right).
		\end{equation*}
		Inserting this into the displayed long third equation of our system at the beginning of the proof yields a polynomial equation in $ t $ of degree two with leading coefficient
		\begin{equation*}
			\frac{(a_3 b_1 - a_1 b_3)(a_1 b_2 - a_2 b_1)}{2(a_2 b_3 - a_3 b_2)},
		\end{equation*}
		not vanishing due to being in the case where none of the factors is zero. Taking a closer look at this quadratic polynomial shows that one solution is $ 0 $ and the other one has height $ \Hc(t) \ll H^{50} $. From this we infer $ \abs{k_1},\abs{k_2},\abs{k_3} \ll H^{58} $ concluding the proof.
	\end{proof}
	
	\begin{remark}
		Let us note that in fact we proved even more in Theorem~\ref{thm:heisenberg3}. Namely, an inspection of the given proof shows that either there are exponents with $ \abs{k_1},\abs{k_2},\abs{k_3} \ll H^2 $ or otherwise we get for \emph{all} solutions $ \abs{k_1},\abs{k_2},\abs{k_3} \ll H^{58} $.
		Furthermore we note that in the case of only two factors $ A_1^{k_1} A_2^{k_2} = I_3 $ we obtain the bound $ \abs{k_1},\abs{k_2} \ll H $.
	\end{remark}
	
	\section{The ABC-identity-problem for symmetric $ (2 \times 2) $-matrices}
	\label{sec:3abc}
	
	We will now discuss an ABC-problem-like topic.
	For the proof of the following result, we need the below given auxiliary statement due to Evertse and Gy\H{o}ry which can be found as Lemma 6.1.1 in \cite{evertse-gyory-2022}.
	It denotes by $h(U)$ the logarithmic height of the set of entries of $U$ and by $h([U,b])$ the logarithmic height of the set of entries of $U$ and $b$ together.
	
	\begin{lemma}
		\label{lem:EGbound}
		Let $U \in \ZZ^{m \times n}$ and $b \in \ZZ^m$.
		The $\ZZ$-module  of $y \in \ZZ^n$ with $Uy = 0$ is generated by vectors in $\ZZ^n$ of logarithmic height at most  $mh(U) + \frac{1}{2}m\log m$.
		Moreover, if we assume that $Uy = b$ is solvable in $\ZZ^n$, then it has a solution $y \in \ZZ^n$ with $h(y) \leq mh([U,b]) + \frac{1}{2}m\log m$.
	\end{lemma}
	
	Our result is now the following:
	
	\begin{theorem}
		\label{thm:decidable}
		Let $A_1, A_2, A_3$ be three symmetric $(2 \times 2)$-matrices with real algebraic number entries of bounded height $h(A_i) \leq H$. Assume that the equation
		\begin{equation}
			\label{eq:mtrxeq}
			A_1^{k_1} A_2^{k_2} = A_3^{k_3}
		\end{equation}  
		has a solution in non-zero integers $k_1, k_2, k_3$. Then there is an effectively computable constant $C$ and exponents $ k_1, k_2, k_3 \leq C$ such that equation \eqref{eq:mtrxeq} holds.\\
		In other words: there is an algorithm which decides for any given symmetric matrices $A_1, A_2, A_3$ with real algebraic number entries whether equation \eqref{eq:mtrxeq} has a solution or not and, in case it has solutions, provides us an explicit one.
	\end{theorem}
	
	\begin{proof}
		Let $A_1, A_2, A_3$ be three symmetric $(2 \times 2)$-matrices with real algebraic numbers as entries. By the spectral theorem, there are real orthogonal matrices $T_i$ as well as algebraic numbers $\lambda_i, \mu_i$ such that 
		\begin{equation*}
			T_i =
			\begin{pmatrix}
				a_i & b_i\\
				-b_i & a_i
			\end{pmatrix}
			\quad \text{and} \quad A_i = T_i
			\begin{pmatrix}
				\lambda_i & 0\\
				0 & \mu_i
			\end{pmatrix}
			T_i^{-1}
		\end{equation*}
		with $a_i^2 + b_i^2 =1$ for $i \in \set{1,2,3}$.
		Then for $ i \in \set{1,2,3} $ and $ k_i \in \ZZ $ we compute
		\begin{equation*}
			A_i^{k_i} =
			\begin{pmatrix}
				a_i^2 \lambda_i^{k_i} + b_i^2 \mu_i^{k_i} & a_i b_i \mu_i^{k_i} - a_i b_i \lambda_i^{k_i} \\
				a_i b_i \mu_i^{k_i} - a_i b_i \lambda_i^{k_i} & a_i^2 \mu_i^{k_i} + b_i^2 \lambda_i^{k_i}
			\end{pmatrix}.
		\end{equation*}
		This gives us
		\begin{align*}
			A_1^{k_1} A_2^{k_2} =\ & \lambda_1^{k_1} \lambda_2^{k_2}
			\begin{pmatrix}
				a_1^2 a_2^2 + a_1 a_2 b_1 b_2 & -a_1 b_1 b_2^2 - a_1^2 a_2 b_2 \\
				-a_1 a_2^2 b_1 - a_2 b_1^2 b_2 & a_1 a_2 b_1 b_2 + b_1^2 b_2^2
			\end{pmatrix} \\
			&+ \lambda_1^{k_1} \mu_2^{k_2}
			\begin{pmatrix}
				a_1^2 b_2^2 - a_1 a_2 b_1 b_2 & a_1^2 a_2 b_2 - a_1 a_2^2 b_1 \\
				a_2 b_1^2 b_2 - a_1 b_1 b_2^2 & a_2^2 b_1^2 - a_1 a_2 b_1 b_2
			\end{pmatrix} \\
			&+ \mu_1^{k_1} \lambda_2^{k_2}
			\begin{pmatrix}
				a_2^2 b_1^2 - a_1 a_2 b_1 b_2 & a_1 b_1 b_2^2 - a_2 b_1^2 b_2 \\
				a_1 a_2^2 b_1 - a_1^2 a_2 b_2 & a_1^2 b_2^2 - a_1 a_2 b_1 b_2
			\end{pmatrix} \\
			&+ \mu_1^{k_1} \mu_2^{k_2}
			\begin{pmatrix}
				a_1 a_2 b_1 b_2 + b_1^2 b_2^2 & a_1 a_2^2 b_1 + a_2 b_1^2 b_2 \\
				a_1^2 a_2 b_2 + a_1 b_1 b_2^2 & a_1^2 a_2^2 + a_1 a_2 b_1 b_2
			\end{pmatrix}.
		\end{align*}
		Suppose now that equation \eqref{eq:mtrxeq} has a solution $(k_1, k_2, k_3)$.
		Considering the matrix components of the equation
		\begin{equation*}
			A_1^{k_1} A_2^{k_2} = A_3^{k_3}
		\end{equation*}
		yields four homogenous linear equations in the six variables $ \lambda_1^{k_1} \lambda_2^{k_2} $, $ \lambda_1^{k_1} \mu_2^{k_2}$, $ \mu_1^{k_1} \lambda_2^{k_2}$, $ \mu_1^{k_1} \mu_2^{k_2}$, $ \lambda_3^{k_3}$, $ \mu_3^{k_3} $.
		We build the coefficient matrix by sorting the variables as just mentioned. Then the submatrix consisting of the first four columns, i.e.\ related to $ \lambda_1^{k_1} \lambda_2^{k_2}$, $ \lambda_1^{k_1} \mu_2^{k_2}$, $ \mu_1^{k_1} \lambda_2^{k_2}$, $ \mu_1^{k_1} \mu_2^{k_2} $, has determinant
		\begin{equation*}
			D = - (a_1 a_2 + b_1 b_2)^2 (a_1 b_2 - a_2 b_1)^2.
		\end{equation*}
		
		The special cases $ a_1 a_2 + b_1 b_2 = 0 $ and $ a_1 b_2 - a_2 b_1 = 0 $, respectively, will be handled later.
		Hence we may now assume that $ D \neq 0 $.
		We will use the abbreviations
		\begin{align*}
			c_1 &\coloneqq \frac{(a_1 a_3 + b_1 b_3) (a_2 a_3 + b_2 b_3)}{(a_1 a_2 + b_1 b_2)}, \\
			c_2 &\coloneqq \frac{(a_1 a_3 + b_1 b_3) (a_3 b_2 - a_2 b_3)}{(a_1 b_2 - a_2 b_1)}, \\
			c_3 &\coloneqq \frac{(a_2 a_3 + b_2 b_3) (a_1 b_3 - a_3 b_1)}{(a_1 b_2 - a_2 b_1)}, \\
			c_4 &\coloneqq \frac{(a_1 b_3 - a_3 b_1) (a_2 b_3 - a_3 b_2)}{(a_1 a_2 + b_1 b_2)}.
		\end{align*}
		Then the system of four equations reduces to the system
		\begin{equation}
			\label{eq:sys4}
			\left\{
			\begin{aligned}
				\lambda_1^{k_1} \lambda_2^{k_2} &= c_1 \lambda_3^{k_3} +c_4 \mu_3^{k_3} \\
				\lambda_1^{k_1} \mu_2^{k_2} &= c_2 \lambda_3^{k_3} +c_3 \mu_3^{k_3} \\
				\mu_1^{k_1} \lambda_2^{k_2} &= c_3 \lambda_3^{k_3} +c_2 \mu_3^{k_3} \\
				\mu_1^{k_1} \mu_2^{k_2} &= c_4 \lambda_3^{k_3} +c_1 \mu_3^{k_3}
			\end{aligned}
			\right..
		\end{equation}
		
		First of all suppose $\lambda_1, \lambda_2, \lambda_3, \mu_1, \mu_2, \mu_3 \neq 0$ and $\lambda_3 / \mu_3$ is not a root of unity. Then the system \eqref{eq:sys4} becomes
		\begin{equation}
			\label{eq:sys-v2}
			\left\{
			\begin{aligned}
				\lambda_1^{k_1} \lambda_2^{k_2} \mu_3^{-k_3} - c_1 \left(\frac{\lambda_3}{\mu_3}\right)^{k_3} &= c_4 \\
				\lambda_1^{k_1} \mu_2^{k_2} \mu_3^{-k_3} - c_2 \left(\frac{\lambda_3}{\mu_3}\right)^{k_3} &= c_3 \\
				\mu_1^{k_1} \lambda_2^{k_2} \mu_3^{-k_3} - c_3 \left(\frac{\lambda_3}{\mu_3}\right)^{k_3} &= c_2 \\
				\mu_1^{k_1} \mu_2^{k_2} \mu_3^{-k_3} - c_4 \left(\frac{\lambda_3}{\mu_3}\right)^{k_3} &= c_1
			\end{aligned}
			\right..
		\end{equation}
		Let $S$ be the finite set of places  containing all the primes above $\lambda_i, \mu_i$ as well as the archimedean ones.
		From Gyory's result \cite{gyory-2019}, there exists an effective constant $C'$, depending on $h(c_i)$, such that
		\begin{equation*}
			k_3 h \left(\frac{\lambda_3}{\mu_3}\right) < C',
		\end{equation*}
		provided that $ c_1 c_4 \neq 0 $ or $ c_2 c_3 \neq 0 $.
		Otherwise the system becomes much simpler, namely all equations have the form '$ S $-unit equals constant', and this will be covered below.
		By using some properties of the height and Dobrowolski inequality, it follows $k_3 < C''$ where the constant $C''$ only depends on $H$ and the number field.
		For any of those values for $k_3$, from relation \eqref{eq:sys-v2} we deduce that $ \lambda_1^{k_1} \lambda_2^{k_2} $, $ \lambda_1^{k_1} \mu_2^{k_2} $, $ \mu_1^{k_1} \lambda_2^{k_2} $, $ \mu_1^{k_1} \mu_2^{k_2} $ all equal a constant.
		Let $S'$ be the finite set of places containing those above $\lambda_i, \mu_i$ as well as above these just mentioned constants and all the archimedean ones.
		By combining Dirichlet's $ S $-unit theorem and Lemma~\ref{lem:EGbound}, we obtain the desired result, even in the above claimed 'simpler' case where some $ c_i $ vanish.
		
		Assume now that $\lambda_3 / \mu_3$ is a root of unity and let $r$ be its order.
		From the system \eqref{eq:sys-v2} we deduce for any fixed $i \in \set{0, 1, \ldots, r-1}$ that
		\begin{equation*}
			\lambda_1^{k_1} \lambda_2^{k_2} = C^{(i)} \mu_3^{k_3}
		\end{equation*}
		and analogous formulas for the other three equations.
		Again the result follows by combining Lemma~\ref{lem:EGbound} and Dirichlet's S-unit Theorem.
		
		Next, we consider the case that one of the $\lambda_i, \mu_i$ is zero.
		If either $ \lambda_3 $ or $ \mu_3 $ is zero, then we argue as in the previous paragraph.
		Thus, without loss of generality assume $\lambda_1=0$.
		In this case the first relation in \eqref{eq:sys4} becomes
		\begin{equation*}
			c_1\lambda_3^{k_3} + c_4 \mu_3^{k_3} =0.
		\end{equation*}
		If $\lambda_3 / \mu_3$ is a root of unity, then we can remove the sum as already done in an other case above.
		Otherwise, a theorem due to Mignotte, Shorey and Tijdeman (Theorem~1 in \cite{mignotte-shorey-tijdeman-1984}) gives us the desired result.
		
		We now deal with the case $D=0$.
		In this situation we have two possibilities, either $a_1a_2+b_1b_2 =0$ or  $a_1b_2-a_2 b_1 =0$.
		Let us first assume $a_1b_2-a_2 b_1 =0$ and without loss of generality $b_2 \neq 0$.
		Upon multiplying the relation $a_1^2 + b_1^2 = 1$ by $b_2^2$, we get $b_1^2 = b_2^2$.
		If $b_1=b_2$, then equation \eqref{eq:mtrxeq} becomes
		\begin{equation*}
			D_1^{k_1}D_2^{k_2} = (T_2^{-1} T_3) \cdot D_3^{k_3} \cdot (T_2^{-1} T_3)^{-1}.
		\end{equation*}
		From Lemma~\ref{lem:LvdP-general} one obtains an explicit upper bound for at least one solution $(k_1, k_2, k_3)$.
		By combining properties of the height with the fact that $\lambda_i, \mu_i$ are roots of the characteristic polynomials of the $A_i$, we get a bound depending only on $H$ and the splitting field associated to the product of these characteristic polynomials.
		Hence we end with the desired result.
		If $b_1= -b_2$, then from the relation $a_1b_2-a_2 b_1 =0$ one gets $a_1= -a_2$ and thus $T_1 = -T_2$.
		Therefore equation \eqref{eq:mtrxeq} becomes again
		\begin{equation*}
			D_1^{k_1}D_2^{k_2} = (T_2^{-1} T_3) \cdot D_3^{k_3} \cdot (T_2^{-1} T_3)^{-1}
		\end{equation*}
		and as before we get the desired result.
		If in the last subcase $a_1a_2+b_1b_2 =0$, then we get our desired result by using the same strategy as above.
		Therefore the proof of our theorem is complete.
	\end{proof}

	\section{A general decidability problem}
	\label{sec:4decida}
	
	We first recall a result due to Ahlgren \cite{ahlgren-1999}, which is a quantitative version of Laurent's Theorem \cite{laurent-1989}.
	Let us consider an equation of the shape
	\begin{equation}
		\label{eq:genmultrec}
		\sum_{l=1}^{k} P_l(\textbf{x}) \boldsymbol{\alpha}_l^{\textbf{x}} = 0
	\end{equation} 
	in variables $\textbf{x} = (x_1, \cdots, x_m) \in \ZZ^m$ where $P_l$ are polynomials with coefficients in a number field $K$, and
	\begin{equation*}
		\boldsymbol{\alpha}_l^{\textbf{x}} \coloneqq \alpha_{l,1}^{x_1} \cdot \ldots \cdot \alpha_{l,m}^{x_m}
	\end{equation*}
	with $\alpha_{l,j} \in K^* $ for $ 1 \leq l \leq k $ and $ 1 \leq j \leq m $.
	We put $\Gamma = \set{1,\ldots,k}$ and denote by $\Pc$ a partition of the set $\Gamma$.
	A subset $\lambda \subseteq \Gamma$ occurring in the partition $\Pc$ will be considered as element of $\Pc$.
	Given $\Pc$, consider the system of equations
	\begin{equation}
		\label{eq:parteqs}
		\sum_{l \in \lambda} P_l(\textbf{x}) \boldsymbol{\alpha}_l^{\textbf{x}} = 0 \qquad (\lambda \in \Pc).
	\end{equation} 
	A solution $\textbf{x}$ of \eqref{eq:parteqs} is called $\Pc$-degenerate if a subsum of one of the equations of the system vanishes.
	Otherwise we will say that $\textbf{x}$ is $\Pc$-non-degenerate.
	Let $M(\Pc)$ be the set of $\Pc$-non-degenerate solutions of the system \eqref{eq:parteqs}.
	Let $G(\Pc)$ be the subgroup of $\ZZ^m$ consisting of $\textbf{x}$ such that
	\begin{equation*}
		\boldsymbol{\alpha}_l^{\textbf{x}} = \boldsymbol{\alpha}_t^{\textbf{x}}
	\end{equation*}
	whenever $l$ and $t$ lie in the same set $\lambda$ of $\Pc$.
	Let $d_0$ be the degree of $K$ and for $l \in \Gamma$ let $\delta_l$ be the total degree of the polynomial $P_l$.
	Put
	\begin{equation*}
		A = \sum_{l \in \Gamma} \binom{m+\delta_l}{m} \qquad \text{and} \qquad D= \max(m, A).
	\end{equation*}
	Let $\abs{\cdot}$ be the euclidean norm on $\RR^m$ and define $\log^+ z = \max\setb{\log z, 1}$.
	Let $r$ be the rank of $G(\Pc)$ and $H$ the subspace of $\RR^m$ spanned by $G(\Pc)$.
	Schlickewei and Schmidt \cite{schlickewei-schmidt-2000} have proven the following result:
	
	\begin{theorem}
		\label{thm:schsch}
		If $G(\Pc) = \set{0}$, then
		\begin{equation*}
			\# M(\mathcal{P}) < 2^{35D^3}d_0^{6D^2}.
		\end{equation*}
	\end{theorem}
	
	Ahlgren's result \cite{ahlgren-1999} deals with the case when the group $ G(\Pc)$ is non-trivial.
	More precisely, define 
	\begin{equation*}
		\Sc = \set{\textbf{x} \in \ZZ^m : \abs{\textbf{x}^{\perp}} \leq c_0 \log^+ \abs{\textbf{x}^{H}}}
	\end{equation*}
	where $\textbf{x}^{\perp} $ lies in the orthogonal of $H$ and $\textbf{x}^{H} \in H$ with $\textbf{x} = \textbf{x}^{H} + \textbf{x}^{\perp}$.
	Let $\mathrm{GL}_m(\ZZ)$ be the set of matrices with integers entries whose determinant equals $\pm 1$.
	We put $c_0= B^{12}2^{4D+12}d_0^6$.
	His result is then:
	
	\begin{theorem}
		\label{thm:ahlgr}
		With the notation above, we have
		\begin{equation*}
			M(\Pc) \subseteq \bigcup_{\psi \in \T} \psi(\Sc)
		\end{equation*}
		for $ \T \subseteq \GG$ with $\# \T \leq 2^{36D^3}d_0^{6D^2}$ and $ \GG $ the group of transformations of $\RR^m$ having the form $\psi(\textbf{x}) = \theta (\textbf{x}) + \textbf{u}$ where $\theta \in \mathrm{GL}_m(\ZZ)$ fixes $H$ pointwise and $\textbf{u} \in \ZZ^m$.
	\end{theorem}
	
	Let us consider two diagonalizable, commuting $(l \times l)$-matrices $A_1,A_2$ with algebraic number entries such that their eigenvalues are neither zero nor roots of unity.
	Moreover, let $\D$ be the set of diagonalizable $(l \times l)$-matrices with algebraic entries which commute with $A_1$ as well as $A_2$ and such that its eigenvalues belong to $\Oc_S$, where $S$ is the set of places containing all the archimedean ones as well as those above the eigenvalues of $A_1$ and $A_2$.
	Our result is now the following:
	
	\begin{theorem}
		\label{thm:solpoly}
		Let $K$ be a number field containing all the eigenvalues of $A_1$ and $A_2$, denoted by $\gamma_1,\ldots,\gamma_l$ and $\beta_1,\ldots,\beta_l$, respectively.
		Let $c>1$ be a real number.
		Assume there is a place of $K$, denoted by $\abs{\cdot}$, such that $\abs{\beta_i},  \abs{\gamma_i} <1$ holds for all $ i \in \set{1,\ldots,l} $.
		Furthermore, let $\Gamma_{1,i}, \Gamma_{2,i}$ be the groups generated by $\set{\gamma_i}$ and $\set{\beta_i}$, respectively, and suppose that
		\begin{equation*}
			\Gamma_{1,i} \cap \Gamma_{2,j} = \set{1}
		\end{equation*}
		for all $ i,j $.
		We fix a polynomial $ f $ of the shape
		\begin{equation*}
			f(X_1,X_2,Z) = Z^d + a_1 Z^{d-1} + \ldots + a_d
		\end{equation*}
		where the $a_i = a_i(X_1,X_2)$ are polynomials with coefficients in $K$.
		Moreover, we assume that $f(0,0,Z)$ has no double root, and enlarge $S$ by adding the places above the roots of $f(0,0,Z)$.
		Consider now the matrix equation 
		\begin{equation}
			\label{eq:mtrxpolyeq}
			f(A_1^n, A_2^m, Z) = 0
		\end{equation} 
		in unknowns $(n,m) \in \NN^2$ and $Z \in \D$.
		Then there exist constants $C_1,C_2 > 0$ as well as effective constants $C_3, \Delta >2$ and at most $d \Delta$ exponential Diophantine equations, indexed by $ i $, of the shape
		\begin{equation*}
			\sum_{\vert (t_1,  t_2) \vert \leq C_3}^{} A_{t_1, t_2}^{(i)} \gamma_1^{nt_1} \beta_1^{mt_2} = 0
		\end{equation*}
		where $\abs{(t_1,t_2)} \coloneqq t_1 + t_2$ and $A_{t_1, t_2}^{(i)}$ are algebraic numbers such that, for all solutions $(n,m,Z)$ of \eqref{eq:mtrxpolyeq} with $n > C_1$, $m > C_2$ and $\max\setb{n,m} \leq c \min\setb{n,m}$, the pair $(n,m)$ is a solution of at least one of these equations and all the eigenvalues of $Z$ belong to the algebra finitely generated by the eigenvalues of $ A_1 $ and $ A_2 $.	
	\end{theorem}
	
	\begin{corollary}
		Let $c>1$ be a real number.
		There exist constants $d',C_1,C_2 > 0$ as well as effective constants $ C_4, \Delta >2$ such that for all solutions $(n,m,Z)$ of \eqref{eq:mtrxpolyeq} with $n > C_1$, $m > C_2$ and $\max\setb{n,m} \leq c \min\setb{n,m}$ we have
		\begin{equation*}
			(n,m) \in \bigcup_{i=1}^{d\Delta} \bigcup_{\Pc} \bigcup_{\psi \in \T_{\Pc}^i} \psi(M_{\Pc})
		\end{equation*}
		where $\Pc$ runs over the set of partitions of $\Gamma = \set{1,\ldots,C_3}$ and  $M_{\Pc}$ corresponds to $\Sc$ from Theorem~\ref{thm:ahlgr} for $ \T_{\Pc}^{i} \subseteq \GG$ with $\# \T_{\Pc}^i \leq 2^{36C_4^3}d'^{6C_4^2}$ and $ \GG $ the group of transformations of $\RR^2$ having the form $\psi(\textbf{x}) = \theta (\textbf{x}) + \textbf{u}$ where $\theta \in \mathrm{GL}_2(\ZZ)$ fixes $H_{\Pc}$ (defined as $H$ in Theorem~\ref{thm:ahlgr} by replacing \eqref{eq:genmultrec} by the $i^{th}$ exponential Diophantine equation from Theorem~\ref{thm:solpoly}) pointwise and  $\textbf{u} \in \ZZ^2$. 
	\end{corollary}
	
	\begin{proof}
		It follows immediately by combining Theorem~\ref{thm:ahlgr} and Theorem~\ref{thm:solpoly}.  
	\end{proof}
	
	\begin{remark}
		Using the same technique together with properties of commuting matrices, we may extend the above result to $k$ given diagonalizable, commuting $(l \times l)$-matrices $A_1,\ldots,A_k$ with algebraic entries.
		Moreover, we can replace the $ 0 $ in the right hand side of \eqref{eq:mtrxpolyeq} by any diagonalizable matrix $ B $ which commutes with all the $ A_i $, provided we restrict to solutions $ Z $ that commute with $ B $ as well.
	\end{remark}
	
	Before we proceed with the proof of the above statement, we first recall some preliminary results that will be needed later.
	The lemma below due to Stewart (Lemma~2 in \cite{stewart-1991}) will allow us to get an explicit control on the roots of $f(\gamma_i^n,\beta_i^m,z)$ within a neighborhood of the origin, provided that $n, m$ are large enough.
	
	\begin{lemma}
		\label{lem:stewart}
		Let $f$ be a polynomial with coefficients from the complex numbers and $ \deg f = n \geq 2 $, $M(f)$ the Mahler measure of $f$ and $D(f)$ its discriminant.
		Denote by $\alpha_1,\ldots,\alpha_n$ the roots of $f$ in $\CC$.
		Then for $ z \in \CC $ we get
		\begin{equation*}
			\abs{f(z)} \geq \dfrac{\abs{D(f)}^{1/2}}{n^{(n-1)/2} 2^{n-1}M(f)^{n-2}} \min_i \set{\abs{z-\alpha_i}}.
		\end{equation*}
	\end{lemma}
	
	Corvaja and Zannier \cite{corvaja-zannier-2002}, using Schmidt's subspace theorem, have proven the following result.
	Here we use the standard notation $ H_S $ for the $ S $-height as well as $ h_S $ for the logarithmic $ S $-height.
	
	\begin{lemma}
		\label{lem:corzan}
		Let $K$ be a number field, $S$ a finite set of absolute values of $K$ containing all the archimedean ones, $\mu$ an absolute value from $S$ as well as $\varepsilon$ a positive real number and $N$ a positive integer.
		Finally, let $c_0,\ldots,c_N \in \overline{K}^*$.
		For $\sigma > (N+2)\varepsilon$ there are only finitely many $(N+1)$-tuples $w \coloneqq (w_0,\ldots,w_N) \in (K^*)^{N+1}$ such that the inequalities 
		\begin{itemize}
			\item $h_S(w_i) + h_S(w_i^{-1}) < \varepsilon h(w_i)$ for $i=1,\ldots,N$,
			\item $\abs{c_0 w_0 + \ldots + c_N w_N}_{\mu} < (H(w_0) H_S(w_0)^{N+1})^{-1} \widehat{H}(w)^{-\sigma}$
		\end{itemize} 
		hold and no subsum of the $c_i w_i$ involving $c_0 w_0$ vanishes.
		Here we used the definition $\widehat{H}(w) \coloneqq \prod_{i=0}^{N} H(w_i)$.
	\end{lemma}
	
	\begin{proof}[Proof of Theorem~\ref{thm:solpoly}]
		Let $(n,m,Z) \in \NN^2 \times \D$ be a solution of equation \eqref{eq:mtrxpolyeq}.
		By the definition of $\D$, the matrices $A_1,A_2,Z$ commute with each other and thus are simultaneously diagonalizable (see Lemma~\ref{lem:simul-diag}), i.e.\ there is a non-singular matrix $T$ such that 
		\begin{equation*}
			A_1 = T
			\begin{pmatrix}
				\gamma_1 & 0 & \cdots & 0 \\
				0 & \gamma_2 & \cdots & 0 \\
				\vdots & \vdots & \ddots & \vdots \\
				0 & 0 & \cdots & \gamma_l 
			\end{pmatrix}
			T^{-1}, \qquad A_2 = T
			\begin{pmatrix}
				\beta_1 & 0 & \cdots & 0 \\
				0 & \beta_2 & \cdots & 0 \\
				\vdots & \vdots & \ddots & \vdots \\
				0 & 0 & \cdots & \beta_l 
			\end{pmatrix}
			T^{-1}
		\end{equation*}
		and
		\begin{equation*}
			Z = T
			\begin{pmatrix}
				z_1(n, m) & 0 & \cdots & 0 \\
				0 & z_2(n, m) & \cdots & 0 \\
				\vdots & \vdots & \ddots & \vdots \\
				0 & 0 & \cdots & z_l(n, m) 
			\end{pmatrix}
			T^{-1}
		\end{equation*}
		where $z_i(n,m)$ are the eigenvalues of $Z$.
		Hence equation \eqref{eq:mtrxpolyeq} becomes 
		\begin{equation}
			\label{eq:linesys}
			\left\{
			\begin{aligned}
				z_1(n,m)^d + a_1(\gamma_1^n,\beta_1^m) z_1(n,m)^{d-1} + \ldots + a_d(\gamma_1^n,\beta_1^m) &= 0 \\
				z_2(n,m)^d + a_1(\gamma_2^n,\beta_2^m) z_2(n,m)^{d-1} + \ldots + a_d(\gamma_2^n,\beta_2^m) &= 0 \\
				&\vdots \\
				z_l(n,m)^d + a_1(\gamma_l^n,\beta_l^m) z_l(n,m)^{d-1} + \ldots + a_d(\gamma_l^n,\beta_l^m) &= 0
			\end{aligned}
			\right..
		\end{equation}
		We consider the equation 
		\begin{equation}
			\label{eq:singleeq}
			f(\gamma_1^n,\beta_1^m,z) \coloneqq z_1(n,m)^d + a_1(\gamma_1^n,\beta_1^m) z_1(n,m)^{d-1} + \ldots + a_d(\gamma_1^n,\beta_1^m) = 0.
		\end{equation}
		The remaining equations in the system \eqref{eq:linesys} will be treated analogously.
		For simplifying the notation, we put
		\begin{equation*}
			a_i(n,m) \coloneqq a_i(\gamma_1^n,\beta_1^m) \qquad \text{and} \qquad z_{n,m} \coloneqq z_1(n,m).
		\end{equation*}
		It is not difficult to see that
		\begin{equation*}
			\abs{z_{n,m}} \leq \max\setb{1,\abs{a_1(n,m)} + \ldots + \abs{a_d(n,m)}}.
		\end{equation*}
		Using $\abs{\beta_1},\abs{\gamma_1} <1$, we find an explicit upper bound for $\abs{z_{n,m}}$.
		Hence it follows that $f(0,0,z_{n,m})$ tends to $0$ when $(n,m) \rightarrow \infty$, and we deduce from Lemma~\ref{lem:stewart} that
		\begin{equation*}
			z_{n,m} \in \bigcup_{k=1}^{d} S_k
		\end{equation*}
		for $n, m$ large enough and $S_k \coloneqq \set{z \in \CC : \abs{z-z^{(k)}} < 1/2}$ where $z^{(k)}$ are the roots of $f(0,0,z)$.
		Without loss of generality we assume $z_{n,m} \in S_1$.
		Since $f(\gamma_1^n,\beta_1^m,z_{n,m})=0$ and
		\begin{equation*}
			\frac{\partial f}{\partial z} (0,0,z^{(1)}) \neq 0
		\end{equation*}
		by assumption, it follows from the implicit function theorem (see Theorem~8 in \cite{fuchs-heintze-2021}) that there is a series $ P $ with algebraic coefficients such that  
		\begin{equation*}
			z_{n,m} = P(\gamma_1^n,\beta_1^m) \coloneqq z^{(1)} + \sum_{\abs{(i_1, i_2)} > 0} A_{(i_1,i_2)} \gamma_1^{n i_1} \beta_1^{m i_2}
		\end{equation*}
		for $n, m$ large enough.
		Next we want to show  that $z_{n, m}$ is a multilinear recurrence sequence.
		It follows by Lemma~9 in \cite{fuchs-heintze-2021} that 
		\begin{equation}
			\label{eq:rootheight}
			h(z_{n,m}) \leq c_{1} + c_{2} h(1 : \gamma_1^n : \beta_1^m)
		\end{equation} 
		for some effectively computable constants $c_1,c_2$.
		Let us consider monomials $\mathbf{X}^{\textbf{i}}$ such that $\abs{\textbf{i}} = i_1 + i_{2} \leq B$ and $A_{\textbf{i}} \neq 0$, where $B$ is a fixed positive integer.
		Later we will choose $B$ large enough to fit its purpose.
		Denote by $N$ the number of such monomials, numbered lexicographically.
		Clearly, we have $N \leq \binom{B+2}{B}$.
		In addition, for our goal we may assume $ N \geq 3 $.
		Put $w= (w_0,\ldots,w_N)$ with $w_0 = z_{n,m} - z^{(1)}$ and $w_{\textbf{i}} = \gamma_1^{n i_1} \beta_1^{m i_{2}} $ for $ 0 < i_1 + i_2 \leq B$.
		Now we want to show that all conditions in Lemma~\ref{lem:corzan} are satisfied in order to apply it.
		To do so, we first estimate the values $H(w_0),\widehat{H}(w),H_S(w_0)$ in terms of $\max\setb{\abs{\gamma_1},\abs{\beta_1}}$.
		By relation \eqref{eq:rootheight} combined with properties of the height, we have
		\begin{equation*}
			H(w_0) \leq e^{c_1} H(1: \gamma_1^n : \beta_1^m )^{c_2} \cdot 2 H(z^{(1)}) \leq e^{c_3} H(1: \gamma_1^n : \beta_1^m )^{c_2}
		\end{equation*}
		for a constant $ c_3 > 1 $.
		Combining this inequality with the fact that $H(\gamma_1^n) = H(\gamma_1)^n$ and  $H(\beta_1^m) = H(\beta_1)^m$, we obtain
		\begin{equation*}
			H(w_0) \leq e^{c_3} \max\setb{\abs{\gamma_1},\abs{\beta_1}}^{-(n+m) \frac{c_2\log Y}{-\log  \max\setb{\abs{\gamma_1},\abs{\beta_1}}}},
		\end{equation*}
		where $Y \coloneqq \max\setb{H(\gamma_1),H(\beta_1)}$.
		Hence, for $n,m$ large enough, we finally get
		\begin{equation*}
			H(w_0) \leq \max\setb{\abs{\gamma_1},\abs{\beta_1}}^{-\max(n,m)T}
		\end{equation*}
		with
		\begin{equation*}
			T \coloneqq 2c_3 + \frac{2c_2\log Y}{-\log \max\setb{\abs{\gamma_1},\abs{\beta_1}}} \geq 2.
		\end{equation*}
		From this estimate one can deduce
		\begin{align*}
			\widehat{H}(w) &= H(w_0) \cdot H(w_1) \cdot \ldots \cdot H(w_N) \\
			&\leq \max\setb{\abs{\gamma_1},\abs{\beta_1}}^{-\max(n,m)T} \cdot \max\setb{H(\gamma_1),H(\beta_1)}^{\max(n,m) NB} \\
			&\leq \max\setb{\abs{\gamma_1},\abs{\beta_1}}^{-\max(n, m) \left( T + NB \frac{\log \max\setb{H(\gamma_1),H(\beta_1)}}{-\log \max\setb{\abs{\gamma_1},\abs{\beta_1}}}\right)} \\
			&\leq \max\setb{\abs{\gamma_1},\abs{\beta_1}}^{-\max(n,m)NBL}
		\end{align*}
		where for the last inequality we have set
		\begin{equation*}
			L \coloneqq \frac{T \log \max\setb{H(\gamma_1),H(\beta_1)}}{-\log \max\setb{\abs{\gamma_1},\abs{\beta_1}}}
		\end{equation*}
		and supposed
		\begin{equation*}
			B > \max\setb{3,\frac{2 \log \max\setb{\abs{\gamma_1},\abs{\beta_1}}}{-\log \max\setb{H(\gamma_1),H(\beta_1)}}} \eqqcolon Q.
		\end{equation*}
		Since $z_{n,m},z^{(1)} \in \Oc_S$, it follows that $H_S(w_0) =1$.
		Now we fix $\varepsilon = 1/BN^2$ as well as $\sigma = (N+3)\varepsilon $ and remind that $N \geq 3$.
		Then we have 
		\begin{align*}
			H(w_0)^{-1}&H_S(w_0)^{-(N+1)}\widehat{H}(w)^{-\sigma} \geq \\
			&\geq \max\setb{\abs{\gamma_1},\abs{\beta_1}}^{\max\setb{n,m}T} \cdot \max\setb{\abs{\gamma_1},\abs{\beta_1}}^{\max\setb{n,m}\sigma NBL} \\
			&\geq \max\setb{\abs{\gamma_1},\abs{\beta_1}}^{\max\setb{n,m} \left( T+\frac{L(N+3)}{N} \right)} \\
			&\geq \max\setb{\abs{\gamma_1},\abs{\beta_1}}^{\max\setb{n,m} (T+ 2L)}.
		\end{align*}
		In the next step we want to estimate the value $\abs{-w_0 + A_1w_1 + \ldots + A_N w_N}$.
		Since the series $ P $ converges within a neighborhood of the origin, its general term converges to zero.
		This implies $\abs{A_{\textbf{i}}} \leq \max\setb{1,\max\abs{z^{*}}} \rho^{\abs{\textbf{i}}}$ for some positive value $\rho$ where $ \max\abs{z^*}$ is taken over all zeros of the polynomial $f(0,0,z)$.
		We know that $ \max\setb{\abs{\gamma_1},\abs{\beta_1}}^{\min\setb{n,m}} $ tends to zero when $(n,m) \rightarrow \infty$, which for $n,m$ large enough implies $2 \rho \sqrt{\max\setb{\abs{\gamma_1},\abs{\beta_1}}^{\min\setb{n,m}}} < 1$.
		Hence
		\begin{align*}
			\abs{-w_0 + A_1w_1 + \ldots + A_N w_N} &\leq \\
			&\hspace{-2cm} \leq \max\setb{1,\max\abs{z^{*}}} \sum_{\abs{\textbf{i}} \geq B} \rho^{\abs{\textbf{i}}} \max\setb{\abs{\gamma_1},\abs{\beta_1}}^{\min\setb{n,m}\abs{\textbf{i}}} \\
			&\hspace{-2cm} \leq \max\setb{1,\max\abs{z^{*}}} \sum_{\textbf{i}} 2^{-\abs{\textbf{i}}} \max\setb{\abs{\gamma_1},\abs{\beta_1}}^{\min\setb{n,m}B/2},
		\end{align*}
		where for the last inequality we used the fact $2 \rho \sqrt{\max\setb{\abs{\gamma_1},\abs{\beta_1}}^{\min\setb{n,m}}} < 1$.
		Therefore the inequality
		\begin{equation*}
			\abs{-w_0 + a_1w_1 + \ldots + a_N w_N} < H(w_0)^{-1} H_S(w_0)^{-(N+1)} \widehat{H}(w)^{-\sigma}
		\end{equation*}
		is satisfied if
		\begin{multline}
			\label{eq:satif}
			\max\setb{1,\max\abs{z^*}} \sum_{\textbf{i}} 2^{-\abs{\textbf{i}}} \max\setb{\abs{\gamma_1},\abs{\beta_1}}^{\min\setb{n,m} B/2} \\
			< \max\setb{\abs{\gamma_1},\abs{\beta_1}}^{\max\setb{n,m} (T+2L)}.
		\end{multline} 
		Since $ \max\setb{n,m} \leq c \min\setb{n,m}$, we deduce that inequality \eqref{eq:satif} holds if
		\begin{multline*}
			\max\setb{1,\max\abs{z^*}} \sum_{\textbf{i}} 2^{-\abs{\textbf{i}}} \max\setb{\abs{\gamma_1},\abs{\beta_1}}^{\max\setb{n,m} B/(2c)} \\
			< \max\setb{\abs{\gamma_1},\abs{\beta_1}}^{\max\setb{n,m} (T+ 2L)}.
		\end{multline*}
		Thus, by setting
		\begin{equation*}
			V \coloneqq \max \left( Q, 2c(T+2L)+ \frac{2c \log \left( \max\setb{1,\max\abs{z^*}} \sum_{\textbf{i}} 2^{-\abs{\textbf{i}}} \right)}{-\log\max\setb{\abs{\gamma_1},\abs{\beta_1}}} \right),
		\end{equation*}
		we can take $ B = \lfloor V \rfloor +1$.
		
		We now want to apply Lemma~\ref{lem:corzan}.
		On the one hand side we have $ h(\beta_1^{a_1}\gamma_1^{a_2}) \neq 0 $ for any nonzero integer vector $(a_1,a_2)$ since $\Gamma_{1,1} \cap \Gamma_{2,1} = \set{1}$.
		Therefore $h(w_i) \neq 0$ for all $i \in \set{1,\ldots,N}$.
		On the other hand side $h_S(w_i)=0$ and $h_S(w_i^{-1})=0$ because $w_i \in \Oc_S^*$ for each $i=1,\ldots,N$.
		So the inequality in the first item of Lemma~\ref{lem:corzan} is satisfied for all $i=1,\ldots,N$. 
		By applying Lemma~\ref{lem:corzan}, we conclude that for $n$ and $m$ large enough a subsum of the $A_{\textbf{i}}w_{\textbf{i}}$ involving $-w_0$ vanishes.
		It is clear that there are not more than $\Delta = 2^N$ such subsums.
		From this we get at most $d \Delta$ multi-recurrences of the shape
		\begin{equation}
			\label{eq:solmultrec}
			\sum_{\abs{(i_1,i_2)} \leq B} A_{(i_1,i_2)}^{j} \gamma_1^{n i_1} \beta_1^{m i_2},
		\end{equation} 
		where $ j $ runs through the set $ \set{1,\ldots,d\Delta} $ and the $A_{(i_1,i_2)}^{j}$ are algebraic numbers, such that $z_{n,m}$ equals one of them.
		By inserting a representation of $z_{n,m}$ given by relation \eqref{eq:solmultrec} into equation \eqref{eq:singleeq}, we finally obtain an exponential Diophantine equation of the shape 
		\begin{equation*}
			\sum_{\abs{(i_1,i_2)} \leq C_3} F_{(i_1,i_2)}^{j} \gamma_1^{n i_1} \beta_1^{m i_2} = 0,
		\end{equation*} 
		where the $F_{(i_1,i_2)}^{j}$ are algebraic numbers and $C_3 \coloneqq C+dB$ with $C$ the maximum of the total degrees among the polynomials $a_i$.
		This completes the proof of our theorem.
	\end{proof}

	\section{Special counting results for multiplicatively dependent tuples}
	\label{sec:5counting}
	
	In this final section we give some bounds on the number of multiplicatively dependent tuples of symmetric matrices.
	Note that symmetric matrices do not form a semigroup and therefore results using this property cannot be applied.
	
	\subsection{Multiplicatively dependent pairs}
	
	We denote by $ \D_n(\ZZ;H) $ the set of diagonalizable $ (n \times n) $-matrices $ A $ with integer entries, $ \Hc(A) \ll H $ and $ \det A \neq 0 $.
	Moreover, by $ \D_{n,2}^*(\ZZ;H) $ we denote the set of multiplicatively dependent pairs of matrices from $ \D_n(\ZZ;H) $ such that none of the two components is a root of the identity matrix.
	If $ A $ is an $ (n \times n) $-matrix with integer entries such that the characteristic polynomial has a multiple zero (in the algebraic closure), then the discriminant of this polynomial is zero.
	Thus there are at most $ \ll H^{n^2-1} $ possibilities for choosing the entries of $ A $ in this situation.
	On the other hand, if the characteristic polynomial has no multiple zero, then the matrix is diagonalizable.
	Katznelson \cite{katznelson-1993} proved that the number of singular matrices with bounded entries is $ \ll H^{n^2-n} \log H $.
	We emphasize at this point that all constants implied by $ \ll $ may only depend on the dimension $ n $.
	Overall we get $ \# \D_n(\ZZ;H) \gg H^{n^2} $.
	
	Furthermore we denote by $ \Sc_n(\ZZ;H) $ the subset of $ \D_n(\ZZ;H) $ consisting of symmetric matrices and by $ \Sc_{n,2}^*(\ZZ;H) $ the subset of $ \D_{n,2}^*(\ZZ;H) $ where both components of the pair are symmetric matrices.
	Since Eskin and Katznelson \cite{eskin-katznelson-1995} proved that the number of singular symmetric matrices with bounded entries is $ \ll H^{(n^2-n)/2} \log H $, we get $ \# \Sc_n(\ZZ;H) \gg H^{(n^2+n)/2} $.
	
	\begin{lemma}
		\label{lem:offdiagonal}
		Let $ a_{ij} $ for $ i,j \in \set{1,\ldots,n} $ with $ i \neq j $ as well as $ \lambda_1,\ldots,\lambda_n $ be given elements from an algebraically closed field $ K $. Then there exists an $ (n \times n) $-matrix over this field with eigenvalues $ \lambda_1,\ldots,\lambda_n $ and whose $ (i,j) $-entry is equal to $ a_{ij} $ for $ i,j \in \set{1,\ldots,n} $ with $ i \neq j $. The number of such matrices is bounded above by $ n^{n^2} $.
	\end{lemma}
	
	\begin{proof}
		The existence and finiteness statement is proven as Theorem~4 in \cite{friedland-1972}. A precise reading of the proofs given there yields a bound $ \leq n^{n^2} $.
	\end{proof}
	
	\begin{theorem}
		\label{thm:boundsymm}
		With the above notation, we have
		\begin{equation*}
			H^{(n^2+n)/2} \ll \# \Sc_{n,2}^*(\ZZ;H) \ll H^{n^2} (\log H)^{2+n}.
		\end{equation*}
	\end{theorem}
	
	\begin{proof}
		We start with proving the upper bound.
		First we fix an arbitrary matrix $ A \in \Sc_n(\ZZ;H) $. For this we have $ \ll H^{(n^2+n)/2} $ possibilities.
		Once $ A $ is fixed, also the eigenvalues $ \lambda_1,\ldots,\lambda_n $ of $ A $ are fixed.
		It remains to count how many matrices $ B $ exist such that $ A^k B^l = I_n $, where we can assume that $ kl \neq 0 $.
		By Theorem~\ref{thm:case2} we may assume that $ \abs{k}, \abs{l} \ll \log H $.
		Denoting the eigenvalues of $ B $ by $ \mu_1,\ldots,\mu_n $, the numbers $ \mu_1^{-l},\ldots,\mu_n^{-l} $ must be a permutation of $ \lambda_1^k,\ldots,\lambda_n^k $.
		There are $ \ll \log H $ possibilities for $ k $ and $ l $, respectively, and again $ \ll \log H $ possible choices for each of the $ l $-th roots, whereas the number of permutations only depends on the dimension $ n $.
		Hence for each matrix $ A $ there are at most $ \ll (\log H)^{2+n} $ possibilities for the eigenvalues of $ B $.
		If we fix now the $ n^2-n $ off-diagonal entries of $ B $, which can be done in $ \ll H^{(n^2-n)/2} $ ways respecting symmetry, Lemma~\ref{lem:offdiagonal} states that then there are not more than $ n^{n^2} $ possibilities to complete the matrix $ B $.
		Thus, in summary, there are at most $ \ll H^{n^2} (\log H)^{2+n} $ pairs with the requested properties.
		
		Now we will prove the lower bound.
		As already mentioned above there are $ \gg H^{(n^2+n)/2} $ matrices in $ \Sc_n(\ZZ;H) $.
		From these we have to remove those matrices which satisfy $ A^k = I_n $ for some $ k > 0 $.
		In this exceptional case, by Lemma~\ref{lem:aki}, all eigenvalues of $ A $ are roots of unity of degree at most $ n $. Therefore their number is bounded by a constant depending only on $ n $.
		Fixing in addition the $ n^2-n $ off-diagonal entries of the symmetric matrix $ A $, Lemma~\ref{lem:offdiagonal} yields that there are $ \ll H^{(n^2-n)/2} $ matrices to remove.
		So there still remain $ \gg H^{(n^2+n)/2} $ choices for $ A $.
		Putting $ B = A $ as well as $ k=1 $ and $ l=-1 $ concludes the proof.
	\end{proof}
	
	\begin{remark}
		\label{rem:diagonal}
		With an analogous proof as for Theorem~\ref{thm:boundsymm} we get
		\begin{equation*}
			H^{n^2} \ll \# \D_{n,2}^*(\ZZ;H) \ll H^{2n^2-n} (\log H)^{2+n}.
		\end{equation*}
		This lower bound improves the lower bound from Theorem~2.5 in \cite{habegger-ostafe-shparlinski-} for the case $ s=2 $, whereas the upper bound is weaker than the corresponding bound in \cite{habegger-ostafe-shparlinski-}.
	\end{remark}
	
	\begin{remark}
		\label{rem:numberfields}
		A natural question is whether or not the results can be transfered to matrices defined over number fields instead of rational integers.
		Here one can replace $ H $ by $ N(H) $, which describes the number of elements in the considered number field having height not exceeding $ H $. Note that $ N(H) $ is always finite due to Northcott's theorem.
		Since the bound for the number of singular matrices due to Katznelson for our purposes could be replaced by the simple fact that $ \det A = 0 $ defines an algebraic variety, indeed we get corresponding bounds in the number field case.
	\end{remark}
	
	\subsection{Multiplicatively dependent $ s $-tuples of symmetric $ (2 \times 2) $-matrices}
	
	Similar to the notation above, we denote by $ \Sc_{2,s}^*(\ZZ;H) $ the set of multiplicatively dependent $ s $-tuples of matrices from $ \Sc_2(\ZZ;H) $ such that no subtuple, i.e.\ $ t $-tuple with $ t<s $ arising from cancelling components from an $ s $-tuple, is multiplicatively dependent.
	Moreover, $ \Sc_2(\ZZ;H;d) $ denotes the subset of $ \Sc_2(\ZZ;H) $ containing those matrices having determinant $ d $.
	
	\begin{lemma}
		\label{lem:divisor}
		Let $ \tau(m) $ denote the number of positive divisors of the positive integer $ m $. Then we have
		\begin{equation*}
			\tau(m) \ll m^{o(1)}.
		\end{equation*}
	\end{lemma}
	
	\begin{proof}
		This follows from Theorem~317 in \cite{hardy-wright-1979} where $ o(1) $ can be chosen as $ \frac{1}{\log \log m} $.
	\end{proof}
	
	\begin{lemma}
		\label{lem:givenDet}
		We have the bound
		\begin{equation*}
			\# \Sc_2(\ZZ;H;d) \ll H^{1+o(1)}.
		\end{equation*}
	\end{lemma}
	
	\begin{proof}
		Let
		\begin{equation*}
			\begin{pmatrix}
				a & b \\ b & c
			\end{pmatrix}
			\in \Sc_2(\ZZ;H;d).
		\end{equation*}
		Then we have $ ac = b^2 + d $.
		First, there are $ \ll H $ possibilities for fixing the entry $ b $.
		As soon as $ b $ is fixed, $ a $ must divide $ b^2 + d $.
		Therefore there are $ \ll \tau(\abs{b^2+d}) $ choices possible.
		Using Lemma~\ref{lem:divisor}, for $ a $ we are restricted to $ \ll H^{o(1)} $ possibilities.
		Finally, $ c $ is now also fixed by $ ac = b^2 + d $ and we are done.
		Note that for the last argument we used that $ b^2 + d \neq 0 $.
		But $ b^2 + d = 0 $ is only possible for at most two values of $ b $ and furthermore at least one of $ a $ and $ c $ also has two be $ 0 $ in that case.
		Thus, for both of those values for $ b $, one of $ a $ and $ c $ is zero and for the other entry there are $ \ll H $ possiblities; this does not exceed the previous bound.
	\end{proof}
	
	\begin{lemma}
		\label{lem:numthlemma}
		Let $ Q > 0 $ be an integer and $ U \geq 1 $ a real number. Denoting by $ F(Q,U) $ the number of positive integers $ u \leq U $ whose prime divisors also divide $ Q $, we have
		\begin{equation*}
			F(Q,U) = (QU)^{o(1)}.
		\end{equation*}
	\end{lemma}
	
	\begin{proof}
		This is Lemma~3.1 in \cite{habegger-ostafe-shparlinski-}.
	\end{proof}
	
	\begin{proposition}
		\label{prop:upperbound}
		With the above notation, we have
		\begin{equation*}
			\# \Sc_{2,s}^*(\ZZ;H) \ll H^{2s+o(1)}.
		\end{equation*}
		The implied constant depends on $ s $.
	\end{proposition}
	
	\begin{proof}
		From $ A_1^{k_1} \cdots A_s^{k_s} = I_2 $ we can deduce by the multiplicativity of the determinant that
		\begin{equation}
			\label{eq:detproduct}
			\prod_{i \in \I} (\det A_i)^{\abs{k_i}} = \prod_{j \in \J} (\det A_j)^{\abs{k_j}}
		\end{equation}
		with $ \I \cap \J = \emptyset $ as well as $ \I \cup \J = \set{1,\ldots,s} $ and $ \abs{k_m} > 0 $ for all $ m \in \set{1,\ldots,s} $ since we assume that no subtuple is multiplicatively dependent.
		As the number of partitions of $ \set{1,\ldots,s} $ only depends on $ s $, we may consider $ \I $ and $ \J $ as fixed in what follows.
		Let us write $ I = \# \I $ and $ J = \# \J $.
		Without loss of generality we can assume $ J \leq I $.
		Now we will count tuples of matrices for which \eqref{eq:detproduct} is possible.
		First, we fix the $ J $ matrices $ A_j $ for $ j \in \J $ which can be done in $ \ll H^{3J} $ ways.
		Put
		\begin{equation*}
			Q = \prod_{j \in \J} \abs{\det A_j}.
		\end{equation*}
		Obviously, the prime factors of $ \det A_i $ for $ i \in \I $ are among those of $ Q $, by \eqref{eq:detproduct}, and $ \abs{\det A_i} \leq CH^2 $ for $ i \in \I $ and a constant $ C $, independent of $ H $.
		Therefore each $ \det A_i $ for $ i \in \I $ can take at most $ 2 F(Q,CH^2) $ values.
		With $ Q \ll H^{2J} $ we get from Lemma~\ref{lem:numthlemma} that $ 2 F(Q,CH^2) \ll H^{o(1)} $.
		Applying Lemma~\ref{lem:givenDet} to each of these values yields that each matrix $ A_i $ with $ i \in \I $ can take only $ \ll H^{1+o(1)} $ values.
		Hence for the $ I $ matrices $ A_i $ with $ i \in \I $ there are $ \ll H^{I+o(1)} $ possibilities.
		Putting things together, in summary there are $ \ll H^{3J+I+o(1)} $ options to choose the $ s $ matrices $ A_1,\ldots,A_s $.
		Since $ I+J=s $ and $ J \leq I $, we have $ 3J+I \leq 2s $, concluding the proof.
	\end{proof}
	
	\begin{proposition}
		\label{prop:lowerbound}
		With the above notation, we have
		\begin{equation*}
			\# \Sc_{2,s}^*(\ZZ;H) \gg
			\begin{cases}
				H^{s+o(1)}, & \text{ if } s \text{ is even,} \\
				H^{s-1}, & \text{ if } s \text{ is odd.}
			\end{cases}
		\end{equation*}
		The implied constant depends on $ s $.
	\end{proposition}
	
	\begin{proof}
		Denote by $ \X_2(\ZZ;K) $ the subset of $ \Sc_2(\ZZ;K) $ consisting of the diagonal matrices with non-zero diagonal elements.
		Obviously, it is $ \# \X_2(\ZZ;K) \gg K^2 $.
		Note that the product of two elements from $ \X_2(\ZZ;K) $ is again an element in $ \X_2(\ZZ;K) $ and that the multiplication in $ \X_2(\ZZ;K) $ is commutative.
		In total, there are $ \gg K^{2s} $ tuples in $ \X_2(\ZZ;K)^s $.
		Furthermore, we want to restrict ourselves to tuples $ (B_1,\ldots,B_s) $ where each determinant $ \det B_i $ has a prime factor not appearing in $ \prod_{j \neq i} \det B_j $.
		For this purpose, $ \det B_i $ has to avoid not more than
		\begin{equation*}
			F\left( \prod_{j \neq i} \abs{\det B_j}, C'K^2 \right) \ll K^{o(1)}
		\end{equation*}
		values, and since the diagonal entries have to divide the determinant, by Lemma~\ref{lem:divisor}, there exist $ \ll K^{o(1)} $ non-admissible choices for $ B_i $.
		Fixing $ B_j $ for $ j \neq i $ and then choosing $ B_i $ with this property, yields that we have to throw away $ \ll K^{2s-2+o(1)} $ tuples.
		Thus, discarding the unsuitable tuples, we see that there still remain $ \gg K^{2s} $ tuples $ (B_1,\ldots,B_s) \in \X_2(\ZZ;K)^s $ such that all $ s $ matrices are invertible and each determinant $ \det B_i $ has a prime factor not appearing in $ \prod_{j \neq i} \det B_j $.
		We denote this set of suitable $ s $-tuples by $ \T $.
		
		Let us first assume that $ s = 2r $ is even.
		We put $ K = \sqrt{H} $.
		Then for any tuple $ (B_1,\ldots,B_s) \in \T $ we set
		\begin{align*}
			A_{2i-1} &\coloneqq B_{2i-1} B_{2i} \\
			A_{2i} &\coloneqq B_{2i+1} B_{2i}
		\end{align*}
		for $ i \in \set{1,\ldots,r} $ with the definition $ B_{s+1} \coloneqq B_1 $.
		Since
		\begin{equation*}
			A_1 A_2^{-1} \cdot \ldots \cdot A_{2r-1} A_{2r}^{-1} = I_2,
		\end{equation*}
		the tuple $ (A_1,\ldots,A_s) $ is multiplicatively dependent.
		Note that each $ B_i $ appears in the definition of exactly two $ A_j $ and that they are assigned in a cyclic way.
		Thus no proper subtuple of $ (A_1,\ldots,A_s) $ can be multiplicatively dependent because each $ \det B_i $ contains an individual prime factor.
		Indeed, if $ A_j $ is skipped in a minimal multiplicatively dependent subtuple and $ A_{j+1} $ is still contained, then $ B_{j+1} $ appears exactly once, for which its exponent must be zero due to the individual prime factor of its determinant, a contradiction.
		
		Now we have to take into account that different tuples $ (B_1,\ldots,B_s) $ could lead to the same tuple $ (A_1,\ldots,A_s) $.
		Hence, the next step is to bound the number of $ (B_1,\ldots,B_s) $ leading to the same $ (A_1,\ldots,A_s) $.
		From the definition of $ A_1 $, we deduce that each diagonal entry of $ B_1 $ has to divide the corresponding diagonal entry of $ A_1 $.
		By Lemma~\ref{lem:divisor}, this means that there are $ \ll K^{o(1)} $ possibilities for $ B_1 $.
		In addition, if $ (A_1,\ldots,A_s) $ and $ B_1 $ are fixed, then $ B_2,\ldots,B_s $ are as well.
		Hence there lead $ \ll K^{o(1)} $ different $ (B_1,\ldots,B_s) \in \T $ to the same $ (A_1,\ldots,A_s) $.
		Overall, we get
		\begin{equation*}
			\# \Sc_{2,s}^*(\ZZ;H) \gg K^{2s+o(1)} = H^{s+o(1)}
		\end{equation*}
		for even $ s $.
		
		Finally, we assume that $ s = 2r+1 $ is odd.
		Again we put $ K = \sqrt{H} $.
		Similar as above, for any tuple $ (B_1,\ldots,B_{s-1}) \in \T' $, where $ \T' $ is generated from $ \T $ by throwing away the last component, we set
		\begin{align*}
			A_{2i-1} &\coloneqq B_{2i-2} B_{2i-1} \\
			A_{2i} &\coloneqq B_{2i} B_{2i-1}
		\end{align*}
		for $ i \in \set{1,\ldots,r} $ and $ A_{2r+1} = B_{2r} $ with the definition $ B_0 \coloneqq I_2 $.
		Since
		\begin{equation*}
			A_1 A_2^{-1} \cdot \ldots \cdot A_{2r-1} A_{2r}^{-1} A_{2r+1} = I_2,
		\end{equation*}
		the tuple $ (A_1,\ldots,A_s) $ is multiplicatively dependent.
		With the same arguments as above, we show that no subtuple is multiplicatively dependent and that different $ (B_1,\ldots,B_{s-1}) $ generate different $ (A_1,\ldots,A_s) $ in this case.
		Thus we get
		\begin{equation*}
			\# \Sc_{2,s}^*(\ZZ;H) \gg K^{2s-2} = H^{s-1}
		\end{equation*}
		for odd $ s $.
		This concludes the proof.
	\end{proof}
	
	\begin{remark}
		The construction of only diagonal matrices in the above proof looks more restrictive than it is.
		It is necessary that the product of two symmetric matrices $ B_i $ is again symmetric.
		To ensure this, for each element above the diagonal we get a polynomial equation in the matrix entries.
		Generically, each such equation will decrease the degrees of freedom by one.
		Thus there remain as many degrees of freedom as diagonal entries exist.
		For this reason we may focus our construction on diagonal matrices in the lower bound.
	\end{remark}
	
	\begin{remark}
		\label{rem:islike}
		The proofs of Proposition~\ref{prop:upperbound} and Proposition~\ref{prop:lowerbound} followed some ideas of \cite{habegger-ostafe-shparlinski-}.
	\end{remark}

	\section*{Acknowledgement}
	
	This research was funded in whole or in part by the Austrian Science Fund (FWF) I6750-N and P34763-N.
	Moreover, the authors are grateful to Alina Ostafe for helpful discussions around the topic of Section~\ref{sec:5counting}.

\end{document}